\documentclass[10pt]{amsart}

\makeatletter
\def\blfootnote{\gdef\@thefnmark{}\@footnotetext}
\makeatother

\usepackage{epsfig}
\usepackage{graphics}
\usepackage{dcpic, pictexwd}

\usepackage[colorlinks=true,
                    linkcolor=blue,
                    urlcolor=blue,
                    citecolor=blue,
                    anchorcolor=blue]{hyperref}
\usepackage{mathtools}
\usepackage{amsmath,amssymb}
\theoremstyle{plain}

\newtheorem*{theorem*}{Theorem}
\newtheorem{theorem}{Theorem}[section]

\newtheorem{lemma}[theorem]{Lemma}

\newtheorem{corollary}[theorem]{Corollary}

\theoremstyle{remark}

\theoremstyle{Acknowledgments}

\newtheorem*{main}{Main Theorem}

\theoremstyle{definition}

\def\mod{{\rm Mod}}

\begin{document}
\blfootnote{\textup{2000} \textit{Mathematics Subject Classification}:
57N05, 20F38, 20F05}
\newenvironment{prooff}{\medskip \par \noindent {\it Proof}\ }{\hfill
$\square$ \medskip \par}
    \def\sqr#1#2{{\vcenter{\hrule height.#2pt
        \hbox{\vrule width.#2pt height#1pt \kern#1pt
            \vrule width.#2pt}\hrule height.#2pt}}}
    \def\square{\mathchoice\sqr67\sqr67\sqr{2.1}6\sqr{1.5}6}
\def\pf#1{\medskip \par \noindent {\it #1.}\ }
\def\endpf{\hfill $\square$ \medskip \par}
\def\demo#1{\medskip \par \noindent {\it #1.}\ }
\def\enddemo{\medskip \par}
\def\qed{~\hfill$\square$}

 \title[Torsion Generators of the Twist Subgroup] {Torsion Generators of the Twist Subgroup}

\author[T{\"{u}}l\.{i}n Altun{\"{o}}z,       Mehmetc\.{i}k Pamuk, and O\u{g}uz Y{\i}ld{\i}z ]{T{\"{u}}l\.{i}n Altun{\"{o}}z,    Mehmetc\.{i}k Pamuk, and O\u{g}uz Y{\i}ld{\i}z}

\address{Department of Mathematics, Middle East Technical University,
 Ankara, Turkey}
\email{atulin@metu.edu.tr}  \email{mpamuk@metu.edu.tr} \email{oguzyildiz16@gmail.com}


\begin{abstract}
We showed that the twist subgroup of the mapping class group of a closed connected nonorientable surface of genus $g\geq13$ 
can be generated by two involutions and an element of order $g$ or $g-1$ depending on whether $g$ is odd or even respectively.
\end{abstract}
\maketitle
  \setcounter{secnumdepth}{2}
 \setcounter{section}{0}
\section{Introduction}

Let $\Sigma_g$ denote a closed connected orientable surface of genus $g$. The mapping class group, $\mod(\Sigma_g)$,
is the group of the isotopy classes of orientation preserving diffeomorphisms of $\Sigma_g$.   
It is a classical result that $\mod(\Sigma_g)$ is generated by finitely many Dehn twists about nonseparating simple closed curves~\cite{de,H,l3}.  
The study of algebraic properties of mapping class group, finding small generating sets, generating sets with particular properties, is an active one leading to interesting developments. 
Wajnryb~\cite{w} showed that $\mod(\Sigma_g)$ can be generated by two elements given as a product of Dehn twists.  As the group is not abelian, this is the smallest 
possible.  Later, Korkmaz~\cite{mk2} showed that one of these generators can be taken as a Dehn twist, he also proved that $\mod(\Sigma_g)$ can be generated by two torsion 
elements. Recently, the third author showed that $\mod(\Sigma_g)$ is generated by two torsions of small orders~\cite{y1}.
 
Generating  $\mod(\Sigma_g)$ by involutions was first considered by McCarthy and Papadopoulus~\cite{mp}.  They showed that the group can be generated 
by infinitely many conjugates of a single involution (element of order two) for $g\geq 3$.     
In terms of generating by finitely many involutions, Luo~\cite{luo} showed that any Dehn twist about a nonseparating simple closed curve 
can be written as a product six involutions, which in turn implies that $\mod(\Sigma_g)$ can be generated by $12g+6$ involutions.  
Brendle and Farb~\cite{bf} obtained a generating set of six involutions for $g\geq3$. Following their work, Kassabov~\cite{ka} showed that 
$\mod(\Sigma_g)$ can be generated by four involutions if $g\geq7$.  Recently, Korkmaz~\cite{mk1} showed that $\mod(\Sigma_g)$ is generated by three involutions 
if $g\geq8$ and four involutions if $g\geq3$. Also, the third author improved his result showing that it is generated by three involutions if $g\geq6$~\cite{y2}.

The main aim of this paper is to find minimal generating sets of torsion elements for a particular subgroup, namely the twist subgroup, of the mapping class groups of nonorientable surfaces.
Let $N_g$ denote a closed connected nonorientable surface of genus $g$. The mapping class group,$\mod(N_g)$,  is defined to be the group of the isotopy classes of 
all diffeomorphisms of $N_g$.  Compared to orientable surfaces less is known about $\mod(N_g)$.  Lickorish~\cite{l1,l2} showed that it is generated by Dehn twists about two-sided simple
closed curves and a so-called $Y$-homeomorphism (or a crosscap slide). Chillingworth~\cite{c} gave a finite generating set for $\mod(N_g)$ that linearly depends on $g$. 
Szepietowski~\cite{sz2} proved that  $\mod(N_g)$ is generated by three elements and by four involutions.

The twist subgroup $\mathcal{T}_g$ of $\mod(N_g)$ is the group generated by Dehn twists about two-sided simple closed curves.
The group $\mathcal{T}_g$ is a subgroup of index $2$ in $\mod(N_g)$ ~\cite{l2}.  Chillingworth~\cite{c} showed that $\mathcal{T}_g$ can be generated by finitely many Dehn twists. 
Stukow~\cite{st2} obtained a finite presentation for $\mathcal{T}_g$ with $(g+2)$ Dehn twist generators. Later  Omori~\cite{om} reduced the number of Dehn twist generators to 
$(g+1)$ for $g\geq4$. If it is not required that all generators are Dehn twists, Du~\cite{du} obtained a generating set consisting of three elements, two involutions and an element of 
order $2g$ whenever $g\geq5$ and odd. 


In the present paper, we prove that $\mathcal{T}_g$ can  be generated 
by two involutions and an element of order $g$ or $g-1$ depending on the parity of $g$ (see Theorems~\ref{rodd} and \ref{reven}).  
\begin{main}
The twist subgroup $\mathcal{T}_{g}$ can be generated by two involutions and an element of order $g$ or $g-1$ depending on whether $g$ is odd or even respectively.
\end{main}

Before we finish the introduction, let us point out that the twist subgroup $\mathcal{T}_{g}$ admits an epimorphism onto the automorphism group of $H_1(N_g;\mathbb{Z}_2)$ 
preserving the $\pmod{2}$ intersection pairing~\cite{mpin}, which is isomorphic to (see~\cite{mk3} and~\cite{sz3})
\begin{eqnarray}
 \begin{cases} Sp(2h;\mathbb{Z}_2)&\text{if $g=2h+1$,} \\ 
 Sp(2h;\mathbb{Z}_2)\ltimes \mathbb{Z}_{2}^{2h+1}&\text{if $g=2h+2$.}  \end{cases}
\nonumber 
\end{eqnarray}
Hence, the action of mapping classes on $H_1(N_g;\mathbb{Z}_2)$ induces an epimorphism from $\mathcal{T}_g$ to 
$Sp\big(2\lfloor\dfrac{g-1}{2}\rfloor;\mathbb{Z}_2\big)$, which immediately implies the following corollary:
\begin{corollary}
The symplectic group  $Sp\big(2\lfloor\dfrac{g-1}{2}\rfloor;\mathbb{Z}_2\big)$ can be generated by two involutions and an element of order $g$ or $g-1$ 
depending on whether $g$ is odd or even respectively.
\end{corollary}

\medskip

\noindent
\textit{Acknowledgments.} The authors thank  Mustafa Korkmaz for helpful conversations and  Tara Brendle for her  helpful comments on a previous version of this paper.   
The first author was partially supported by the Scientific and Technologic Research Council of Turkey (T\"{U}B\.{I}TAK)[grant number 117F015].


\par  
\section{Background and Results on Mapping Class Groups} \label{S2}
 Let $N_g$ be a closed connected nonorientable surface of genus $g$. 
 Note that the {\textit{genus}} for a nonorientable surface is the number 
 of projective planes in a connected sum decomposition. The {\textit{mapping class group}} 
 $\mod(N_g)$ of the surface $N_g$ is defined to be the group of the isotopy classes of 
 diffeomorphisms $N_g \to N_g$. Throughout the paper we do not distinguish a 
 diffeomorphism from its isotopy class. For the composition of two diffeomorphisms, we
use the functional notation; if $g$ and $h$ are two diffeomorphisms, 
the composition $gh$ means that $h$ acts on $N_g$ first.\\
\indent
A simple closed curve on a nonorientable surface $N_g$ is said to be 
\textit{one-sided} if a regular neighbourhood of it is homeomorphic to 
a M\"{o}bius band. It is called \textit{two-sided} if a regular neighbourhood of 
it is homeomorphic to an annulus. If $a$ is a two-sided simple closed 
curve on $N_g$, to define the Dehn twist $t_a$, we need to fix one of two possible 
orientations on a regular neighbourhood of $a$ (as we did for the 
curve $a_1$ in Figure~\ref{G}). Following ~\cite{mk1} the right-handed 
Dehn twist $t_a$ about $a$ will be denoted by the corresponding capital 
letter $A$.

Now, let us recall the following basic properties of Dehn twists which we use frequently in the remaining of the paper. Let $a$ and $b$ be 
two-sided simple closed curves on $N_g$ and $f\in \mod(N_g)$.
\begin{itemize}
\item \textbf{Commutativity:} If $a$ and $b$ are disjoint, then $AB=BA$.
\item \textbf{Conjugation:} If $f(a)=b$, then $fAf^{-1}=B^{s}$, where $s=\pm 1$ 
depending on whether $f$ is orientation preserving or orientation reversing on a 
neighbourhood of $a$ with respect to the chosen orientation.
\end{itemize}

\begin{figure}[h]
\begin{center}
\scalebox{0.3}{\includegraphics{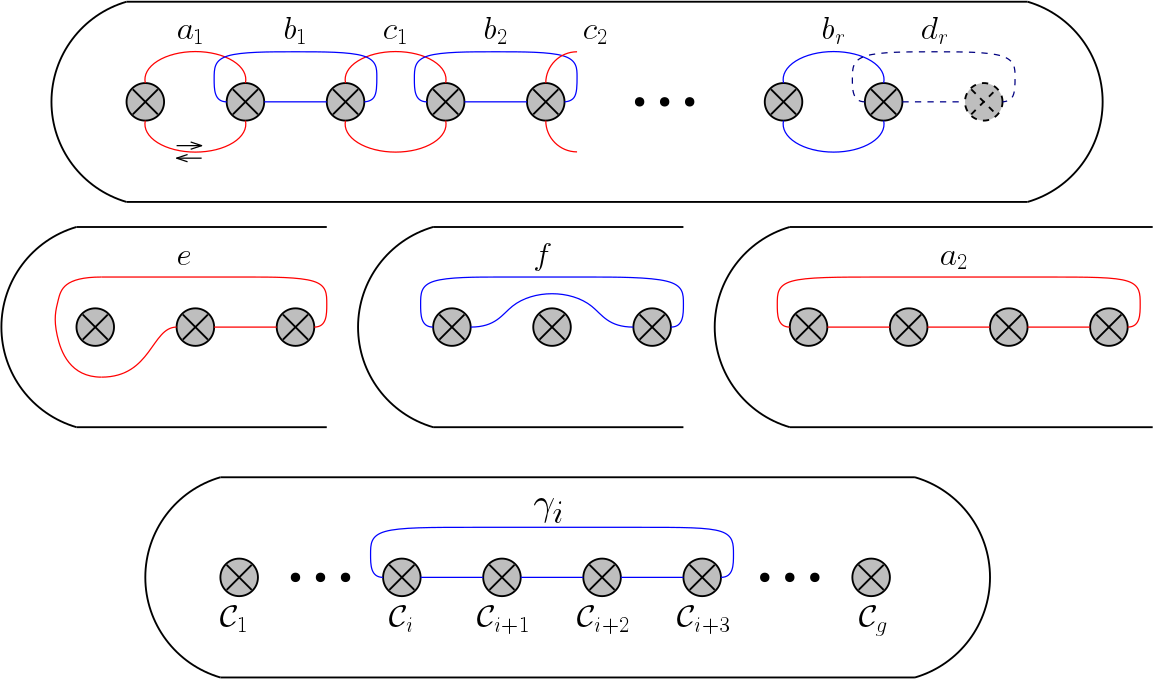}}
\caption{The curves $a_1,a_2,b_i,c_i,e,f$ and $\gamma_i$ on the surface $N_g$.}
\label{G}
\end{center}
\end{figure}
\par
\begin{figure}[h]
\begin{center}
\scalebox{0.17}{\includegraphics{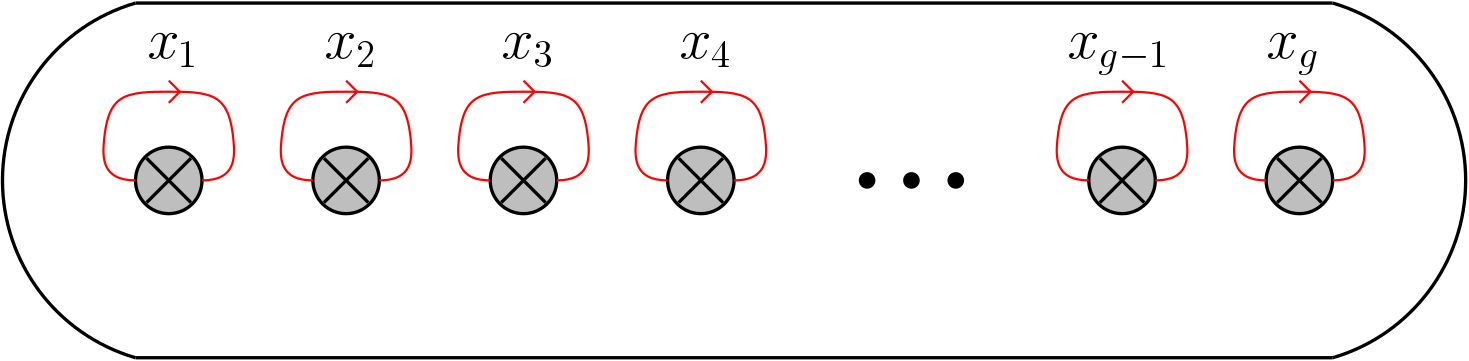}}
\caption{Generators of $H_1(N_g;\mathbb{R})$.}
\label{H}
\end{center}
\end{figure}

Consider the surface $N_g$ shown in Figure~\ref{G}.
The Dehn twist generators of Omori can be given as follows (note that we do not have the curve $d_r$ when $g$ is odd).
\begin{theorem}\cite{om}\label{thm1}
The twist subgroup $\mathcal{T}_g$ is generated by the following $(g+1)$ Dehn twists
\begin{enumerate}
\item $A_1,A_2,B_1,\ldots, B_r$, $C_1,\ldots, C_{r-1}$ and $E$ if $g=2r+1$ and 
\item $A_1,A_2,B_1,\ldots, B_r$, $C_1,\ldots, C_{r-1}$, $D_r$ and $E$ if $g=2r+2$.
\end{enumerate}
\end{theorem}

\noindent
Consider a basis $\lbrace x_1, x_2. \ldots, x_{g-1}\rbrace$ for $H_1(N_g; \mathbb{R})$ 
such that the curves $x_i$ are one-sided and disjoint as in Figure~\ref{H}. It is known that every
 diffeomorphism $f: N_g \to N_g$ induces a linear map 
 $f_{\ast}: H_1(N_g;\mathbb{R}) \to H_1(N_g;\mathbb{R})$. Therefore, one can
  define a homomorphism $D: \mod(N_g) \to \mathbb{Z}_{2}$ by $D(f)=\textrm{det}(f_{\ast})$. 
  The following lemma from~\cite{l1} tells when a mapping class falls into the twist subgroup $\mathcal{T}_g$.

\begin{lemma}\label{lem1} Let $f\in  \mod(N_g)$. Then  $D(f)=1$ if $f\in \mathcal{T}_g$ and
$D(f)=-1$ if $f \not \in \mathcal{T}_g$.
\end{lemma}
\subsection{A generating set for $\mathcal{T}_g$}
We start with presenting a generating set for $\mathcal{T}_g$. The diffeomorphism $T$ is the rotation by $\frac{2\pi}{g}$
or $\frac{2\pi}{g-1}$ as shown on the right hand sides of Figures~\ref{MO} and ~\ref{ME}, respectively. Note that the rotation $T$ satisfies $D(T)=1$, which implies that $T$ belongs to $\mathcal{T}_g$.
\begin{figure}[h]
\begin{center}
\scalebox{0.25}{\includegraphics{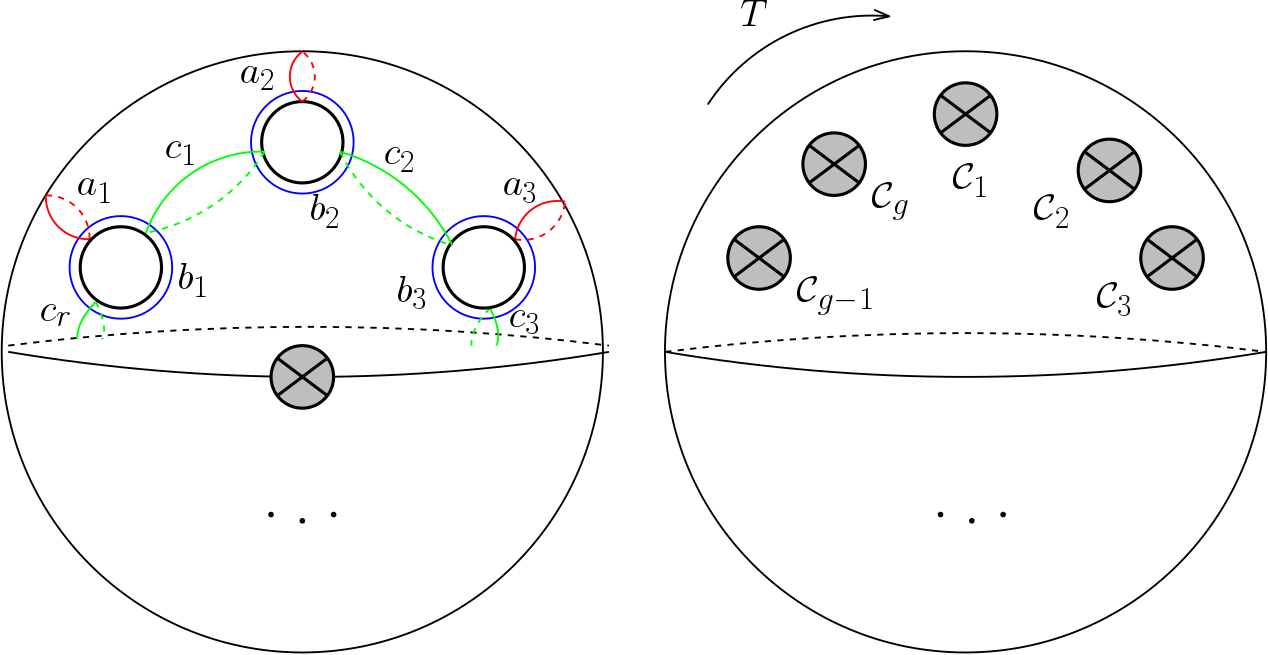}}
\caption{The models for $N_g$ if $g=2r+1$.}
\label{MO}
\end{center}
\end{figure}

\begin{figure}[h]
\begin{center}
\scalebox{0.4}{\includegraphics{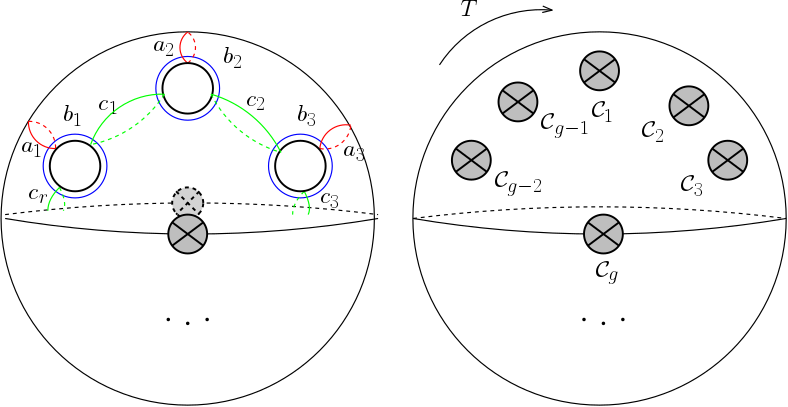}}
\caption{The models for $N_g$ if $g=2r+2$.}
\label{ME}
\end{center}
\end{figure}
\begin{theorem}\label{t1}
The twist subgroup $\mathcal{T}_g$ is generated by the elements 
\begin{enumerate}
\item $T,A_1A_{2}^{-1},B_1B_{2}^{-1}$ and $E$ if $g=2r+1$ and $r\geq 3$,
\item $T,A_1A_{2}^{-1},B_1B_{2}^{-1}, D_r$ and $E$ if $g=2r+2$ and $r\geq 3$.
\end{enumerate}
 \end{theorem}
\begin{proof}
Let $G$ be the subgroup of $\mathcal{T}_g$ generated by the following set 
\[
G=\left\{\begin{array}{lll}
\lbrace T,A_1A_{2}^{-1},B_1B_{2}^{-1},E \rbrace & \textrm{if} & g=2r+1,\\
\lbrace T,A_1A_{2}^{-1},B_1B_{2}^{-1},D_r,E \rbrace & \textrm{if} & g=2r+2,\\
\end{array}\right.
\]
where $r\geq3$.
It follows from Theorem~\ref{thm1} that we only need to prove that $G$ contains the elements $A_1,A_2,B_i$ and $C_j$ shown in Figures~\ref{MO} and~\ref{ME} where $i=1,\ldots,r$ and $j=1,\ldots,r-1$. (We use the explicit homeomorphism constructed in \cite[Section $3$]{st1} to identify the models in these figures.)\\
\noindent
Let $\mathcal{S}$ denote the set of isotopy classes of two-sided non-separating 
simple closed curves on $N_g$. Define a subset $\mathcal{G}$ of $\mathcal{S}\times \mathcal{S}$ 
as 
\[
\mathcal{G} =\lbrace(a,b): AB^{-1}\in G \rbrace.
\]
Using the arguments similar to the proof of ~\cite[Theorem $5$]{mk1}, the set $\mathcal{G}$ satisfies
\begin{itemize}
	\item if $(a,b)\in \mathcal{G}$, then $(b,a)\in \mathcal{G}$ (symmetry),
	\item if $(a,b) \ \textrm{and} \ (b,c)\in \mathcal{G}$, then $(a,c)\in \mathcal{G}$ (transitivity)  and
	\item if $(a,b)\in \mathcal{G}$ and $H\in G$ then $(H(a),H(b))\in \mathcal{G}$ ($G$-invariance).
	\end{itemize}
Hence, $\mathcal{G}$ defines an equivalence relation on $\mathcal{S}$. \\
\noindent
We begin by showing that $B_iC_{j}^{-1}$ is contained in $G$ for all $i,j$. It will follow from the definition of $G$ and from the fact that $T(b_1,b_2)=(c_1,c_2)$, we have $C_1C_{2}^{-1}$
is in $G$ (here, we use the notation $f(a,b)$ to denote $(f(a),f(b))$). Also, by conjugating  $C_1C_{2}^{-1}$ with powers of $T$, one can show that the elements $B_{i}B_{i+1}^{-1}$ and $C_{i}C_{i+1}^{-1}$ are contained in $G$. Moreover, the subgroup $G$ contains the elements $B_{i}B_{j}^{-1}$ and $C_{i}C_{j}^{-1}$ by the transitivity. To start with, since $B_2B_{3}^{-1}\in G$ and it is easy to check that 
\[
B_2B_{3}^{-1}A_2A_{1}^{-1}(b_2,b_3)=(a_2,b_3),
\]
so that $A_2B_{3}^{-1}$ is contained in the subgroup $G$. We have
\[
	 (A_1A_{2}^{-1})(A_2B_{3}^{-1})(B_3B_{2}^{-1})=A_{1}B_{2}^{-1}\in G,
\]
since each of the factors is contained in $G$. Hence, $T(a_1,b_2)=(b_1,c_2)$ implies that $B_1C_{2}^{-1}$ is also in $G$. Now, the subgroup $G$ contains the element
\[
B_1C_{1}^{-1}=(B_1C_{2}^{-1})(C_2C_{1}^{-1}).
\]
Therefore, the elements $B_iC_{i}^{-1}$ is contained in $G$ by conjugating with powers of $T$ for all $i=1,\ldots,r-1$. It follows from the transitivity that $B_iC_{j}^{-1}$ is in $G$. Note that, we have
\begin{itemize}
	\item $(A_1B_{2}^{-1})(B_2C_{1}^{-1})=A_1C_{1}^{-1}$,
	\item$(C_1A_{1}^{-1})(A_1A_{2}^{-1})=C_1A_{2}^{-1}$, and
	\item$(C_2C_{1}^{-1})(C_1A_{1}^{-1})=C_2A_{1}^{-1}$
	\end{itemize}
	from which it follows that the elements $A_1C_{1}^{-1}$, $C_1A_{2}^{-1}$ and $C_2A_{1}^{-1}$ belong to $G$. It can also be shown that
\[
(B_2A_{1}^{-1})(C_1A_{2}^{-1})(C_2A_{1}^{-1})(b_2,a_1)=(d_1,a_1)
\]
and
\[
(A_1B_{2}^{-1})(A_1C_{1}^{-1})(A_1C_{2}^{-1})(A_1B_{2}^{-1})(a_2,a_1)=(d_2,a_1)
\]
so that $G$ contains $D_1A_{1}^{-1}$ and $D_2A_{1}^{-1}$ (here, the curves $d_1$ and $d_2$ are shown in \cite[Figure $1$]{mk1}). Also, we have
\[
(D_2A_{1}^{-1})(A_1C_{1}^{-1})=D_2C_{1}^{-1} \in G.
\]
By similar arguments as in the proof of ~\cite[Theorem $5$]{mk1}, the lantern relation implies the following identity
\[
A_3=(A_2C_{2}^{-1})(D_1A_{1}^{-1})(D_2C_{1}^{-1} ).
\]
Since the subgroup $G$ contains each factor on the right hand side, the element $A_3$ belongs to $G$. It follows from 
\[
B_3=A_3(B_3B_{1}^{-1})A_3(B_1B_{3}^{-1})A_{3}^{-1}
\]
that $B_3$ is also contained in $G$. By conjugating $B_3$ with the powers of $T$, we get $A_1,B_1,C_1,\ldots B_{r-1},C_{r-1}$ and $B_r$
are all contained in $G$. Moreover, 
\[
A_2=(A_2A_{1}^{-1})A_1 \in G.
\]
Therefore, we conclude that $G=\mathcal{T}_g$.
\end{proof}

\begin{figure}[h]
\begin{center}
\scalebox{0.215}{\includegraphics{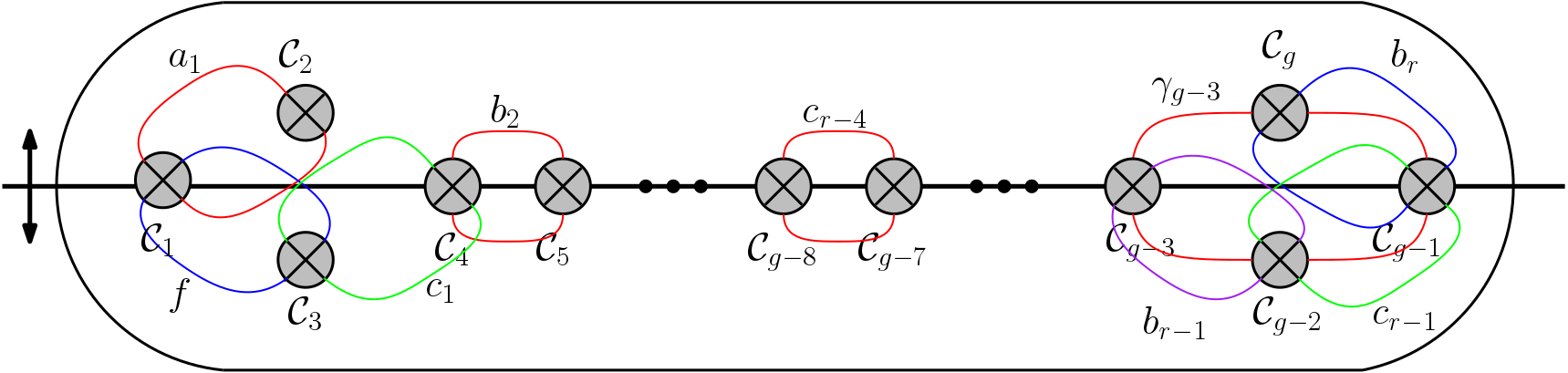}}
\caption{The involution $\sigma$ if $g=2r+1$.}
\label{sigmao}
\end{center}
\end{figure}

\begin{figure}[h]
\begin{center}
\scalebox{0.35}{\includegraphics{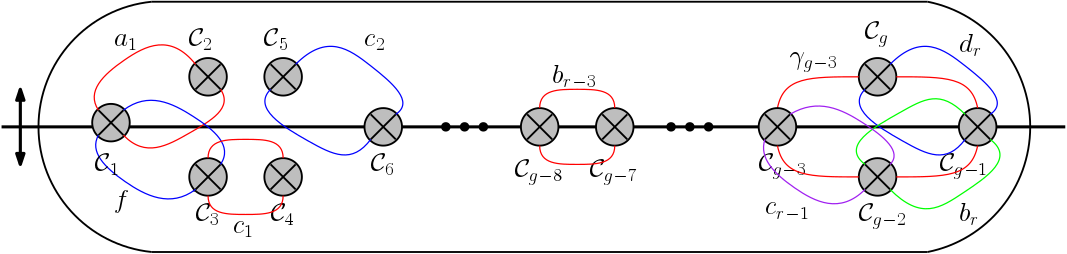}}
\caption{The involution $\sigma$ if $g=2r+2$.}
\label{sigmae}
\end{center}
\end{figure}
\indent

 We consider the surface $N_g$ where $g$-crosscaps are 
distributed on the sphere as in Figures~\ref{sigmao} and~\ref{sigmae}. If $g=2r+1\geq 13$, there is a 
reflection, $\sigma$, of the surface $N_g$ in the $xy$-plane such that 
\begin{itemize}
	\item $\sigma(a_1)=f$, $\sigma(b_r)=c_{r-1}$,
	\item $\sigma(x_2)=x_3$, $\sigma(x_{g})=x_{g-2}$ and
	\item $\sigma(x_i)=x_i$ if $i=4,\ldots,g-3$ or $i=1,g-1$.
\end{itemize}
with reverse orientation as in Figure~\ref{sigmao}. (Recall that $x_i$'s are the generators of $H_1(N_g;\mathbb{R})$ 
as shown in Figure~\ref{H}.) If $g=2r+2\geq 14$, there is a 
reflection, $\sigma$, of the surface $N_g$ in the $xy$-plane such that 
\begin{itemize}
	\item $\sigma(a_1)=f$, $\sigma(b_r)=d_{r}$,
	\item $\sigma(x_2)=x_3$, $\sigma(x_3)=x_4$, $\sigma(x_{g})=x_{g-2}$ and
	\item $\sigma(x_i)=x_i$ if $i=6,\ldots,g-3$ or $i=1,g-1$.
\end{itemize}
with reverse orientation as in Figure~\ref{sigmae}. Note that in both cases the reflection $\sigma$ in in $\mathcal{T}_g$ since $D(\sigma)=1$ for $g\geq13$.

Now, for the remaining part of the paper, let $\Gamma_{i}$ denote the right handed Dehn twist about the curve $\gamma_i$ shown in Figure~\ref{G}.

\begin{theorem}\label{rodd}
For odd $g\geq13$, the twist subgroup $\mathcal{T}_{g}$ is generated by the 
three elements  $T,\sigma$ and $\sigma\Gamma_{g-3}C_{\frac{g-9}{2}}^{-1}$.
\end{theorem}
\begin{proof}
Consider the surface $N_g$ as in Figure~\ref{sigmao}.
Since 
\[
\sigma(\gamma_{g-3})=\gamma_{g-3} \textrm{ and }\sigma(c_{\frac{g-9}{2}})=c_{\frac{g-9}{2}},
\]
and $\sigma$ reverses the orientation of a neighbourhood of a two-sided simple closed curve, we have
\[
\sigma \Gamma_{g-3}\sigma =\Gamma_{g-3}^{-1} \textrm{ and }\sigma C_{\frac{g-9}{2}}\sigma=C_{\frac{g-9}{2}}^{-1}.
\]
Hence, it is easy to verify that  $\sigma\Gamma_{g-3}C_{\frac{g-9}{2}}^{-1}$ is an involution.
Let $H$ be the subgroup of $\mathcal{T}_g$ generated by the following set 
\[
\{T,\sigma, \sigma\Gamma_{g-3}C_{\frac{g-9}{2}}^{-1} \},
\]
where $g\geq13$ and odd. It follows from Theorem~\ref{t1} that we only need to prove that the elements $A_1A_{2}^{-1}$, $B_1B_{2}^{-1}$ and $E$ are contained in the subgroup $H$.\\
Since
\[
\Gamma_{g-3}C_{\frac{g-9}{2}}^{-1}=(\sigma)(\sigma\Gamma_{g-3}C_{\frac{g-9}{2}}^{-1}),	
\]
the element $\Gamma_{g-3}C_{\frac{g-9}{2}}^{-1}$ belongs to $H$.\\
\noindent
It follows from 
\begin{itemize}
\item$T^{13-g}(\gamma_{g-3},c_{\frac{g-9}{2}})=(\gamma_{10},c_{2})$,
\item $T^{-4}(\gamma_{10},c_2)=(\gamma_{6},a_{1})$ and
\item $T^{2}(\gamma_{6},a_1)=(\gamma_{8},c_{1})$
\end{itemize}
that the elements $\Gamma_{10}C_{2}^{-1}$, $\Gamma_{6}A_{1}^{-1}$ and $\Gamma_{8}C_{1}^{-1}$ are in $H$. Since 
\[
(\Gamma_6A_{1}^{-1})(C_2\Gamma_{10}^{-1})(\gamma_6,a_1)=(c_2,a_1),
\]
the element $C_2A_{1}^{-1}$ is in $H$. Since
\begin{itemize}
\item $(\Gamma_6A_{1}^{-1})(A_1C_{2}^{-1})=\Gamma_6C_{2}^{-1}\in H$ and
\item $T^4(\gamma_6,c_2)=(\gamma_{10},c_4)$,
\end{itemize}
$\Gamma_{10}C_{4}^{-1}$ is contained in $H$. Also, we have
\[
(C_4\Gamma_{10}^{-1})( \Gamma_{10}C_{2}^{-1})=C_4C_{2}^{-1}\in H.
\]
 Thus,
$C_3C_{1}^{-1}$, $C_3C_{5}^{-1}$, $B_4B_{2}^{-1}$ and $C_2A_{1}^{-1}$ are contained in $H$ by conjugating $C_4C_{2}^{-1}$ with some powers of $T$. Then, since
\[
(C_4C_{2}^{-1})(B_4B_{2}^{-1})(c_2,a_1)=(b_2,a_1),
\]
$H$ contains $B_2A_{1}^{-1}$. Also, we get 
\[
(C_{2}A_{1}^{-1})(A_1B_{2}^{-1})=C_{2}B_{2}^{-1} \in H. 
\]
Since
\[
C_1C_{3}^{-1}B_2B_{4}^{-1}(c_3,c_5)=(b_4,c_5),
\]
$H$ contains $B_4C_{5}^{-1}$. Then, 
\[
 T^{-4}(b_4,c_5)=(b_2,c_3)
\]
implies that $B_2C_{3}^{-1}\in H$. It follows from the following equalities
\begin{itemize}
\item $C_2C_{3}^{-1}=(C_2B_{2}^{-1})(B_2C_{3}^{-1})$ and
\item $T^{-2}(b_2,c_3)=(b_1,b_{2})$
\end{itemize}
that $B_1B_{2}^{-1}$ belongs to $H$. On the other hand, since
\begin{itemize}
\item $(\Gamma_8C_{1}^{-1})(C_1C_{3}^{-1})(C_3C_{5}^{-1})(C_5B_{4}^{-1})=\Gamma_8B_{4}^{-1}$ and
\item $T^{-7}(\gamma_8,b_4)=(\gamma_1,a_{1})=(a_2,a_1)$
\end{itemize}
that $A_1A_{2}^{-1}$ is contained in $H$.\\
\noindent
Since the elements $T$, $A_1A_{2}^{-1}$ and $B_1B_{2}^{-1}$ are contained in the subgroup $H$, the generators $A_1,B_1,C_1,\ldots, B_{r-1},C_{r-1}, B_r$ are in $H$ by the proof of Theorem~\ref{thm1}. Moreover, 
\[
A_1\sigma(a_1)=A_1(f)=e
\]
leads to $E\in H$. This completes the proof.
\end{proof}
\begin{theorem}\label{reven}
For even $g\geq14$, the twist subgroup $\mathcal{T}_{g}$ is generated by the 
three elements  $T,\sigma$ and $\sigma\Gamma_{g-3}C_{\frac{g-8}{2}}^{-1}$.
\end{theorem}
\begin{proof}
Consider the surface $N_g$ as in Figure~\ref{sigmae}.
Since 
\[
\sigma(\gamma_{g-3})=\gamma_{g-3} \textrm{ and }\sigma(c_{\frac{g-8}{2}})=c_{\frac{g-8}{2}},
\]
and $\sigma$ reverses the orientation of a neighbourhood of a two-sided simple closed curve, we have
\[
\sigma \Gamma_{g-3}\sigma =\Gamma_{g-3}^{-1} \textrm{ and }\sigma C_{\frac{g-8}{2}}\sigma=C_{\frac{g-8}{2}}^{-1}.
\]
Hence, it is easy to show that  $\sigma\Gamma_{g-3}C_{\frac{g-9}{2}}^{-1}$ is an involution.
Let $K$ be the subgroup of $\mathcal{T}_g$ generated by the following set 
\[
\{T,\sigma, \sigma\Gamma_{g-3}C_{\frac{g-8}{2}}^{-1} \},
\]
where $g\geq14$ and even.\\
\noindent
Since
\[
\Gamma_{g-3}C_{\frac{g-8}{2}}^{-1}=(\sigma)(\sigma\Gamma_{g-3}C_{\frac{g-8}{2}}^{-1}),	
\]
the element $\Gamma_{g-3}C_{\frac{g-8}{2}}^{-1}$ is contained in $K$.\\
\noindent
It follows from 
\[
T^{13-g}(\gamma_{g-3},c_{\frac{g-8}{2}})=(\gamma_{10},c_{2})
\]
that the elements $\Gamma_{10}C_{2}^{-1}$ is in the subgroup $K$. Recall that this element also appears in the proof of Theorem~\ref{rodd}. The remaining part of the proof follows as in the proof of the previous theorem. Also, note that the element $D_r$ belongs to $K$ since $\sigma(b_r)=d_r$. This finishes the proof.
\end{proof}


\end{document}